\documentclass[oneside,english]{amsart}
\usepackage[T1]{fontenc}
\usepackage[utf8]{luainputenc}
\usepackage{geometry}
\usepackage{babel}
\usepackage{amsthm}
\usepackage{amstext}
\usepackage{amssymb}
\usepackage{esint}
\usepackage[numbers]{natbib}
\usepackage[unicode=true,pdfusetitle,
 bookmarks=true,bookmarksnumbered=false,bookmarksopen=false,
 breaklinks=false,pdfborder={0 0 1},backref=false,colorlinks=false]
 {hyperref}

\makeatletter

\let\SF@@footnote\footnote
\def\footnote{\ifx\protect\@typeset@protect
    \expandafter\SF@@footnote
  \else
    \expandafter\SF@gobble@opt
  \fi
}
\expandafter\def\csname SF@gobble@opt \endcsname{\@ifnextchar[
  \SF@gobble@twobracket
  \@gobble
}
\edef\SF@gobble@opt{\noexpand\protect
  \expandafter\noexpand\csname SF@gobble@opt \endcsname}
\def\SF@gobble@twobracket[#1]#2{}

\numberwithin{equation}{section}
\numberwithin{figure}{section}
  \theoremstyle{definition}
  \newtheorem{defn}{\protect\definitionname}
  \theoremstyle{plain}
  \newtheorem{lem}{\protect\lemmaname}
\theoremstyle{plain}
\newtheorem{thm}{\protect\theoremname}
  \theoremstyle{plain}
  \newtheorem{assumption}{\protect\assumptionname}
  \theoremstyle{remark}
  \newtheorem{rem}{\protect\remarkname}
  \theoremstyle{remark}
  \newtheorem*{acknowledgement*}{\protect\acknowledgementname}

\makeatother

  \providecommand{\acknowledgementname}{Acknowledgement}
  \providecommand{\assumptionname}{Assumption}
  \providecommand{\definitionname}{Definition}
  \providecommand{\lemmaname}{Lemma}
  \providecommand{\remarkname}{Remark}
\providecommand{\theoremname}{Theorem}

\begin{document}

\title{The Free Boundary of Variational Inequalities with Gradient Constraints}

\author{Mohammad Safdari}
\begin{abstract}
In this paper we prove that the free boundary of the minimizer of
\[
I(v):=\int_{U}\frac{1}{2}|Dv|^{2}-\eta v\, dx,
\]
subject to the pointwise gradient constraint 
\[
|Dv|_{p}\le1,
\]
is as regular as the tangent bundle of the boundary of the domain.
To this end, we study a generalized notion of ridge of a domain in
the plane, which is the set of singularity of the distance function
in the $p\hspace{1bp}$-norm to the boundary of the domain.
\end{abstract}

\maketitle

\section{Introduction}

Let $U$ be a bounded, open set in $\mathbb{R}^{2}$ whose boundary
is at least $C^{2}$.%
\footnote{This implies that $U$ is the union of finitely many components, and
each component is the interior of a simple closed Jordan curve with
finitely many holes. Each hole is itslef, the interior of a simple
closed Jordan curve.%
} Suppose $K\subset\mathbb{R}^{2}$ is a balanced (symmetric with respect
to the origin) closed convex set whose interior contains the origin.
Let 
\[
\gamma_{K}(x):=\inf\{\lambda>0\,\mid\, x\in\lambda K\}
\]
be the gauge function of $K$. Also let 
\[
K^{\circ}:=\{x\,\mid\, x\cdot k\leq1\,\textrm{ for all }k\in K\}
\]
be the polar of $K$. By our assumptions, $\gamma_{K}$ and $\gamma_{K^{\circ}}$
are norms on $\mathbb{R}^{2}$ (see \citet{MR0274683}). Let $d_{K^{\circ}}$
be the metric associated to the norm $\gamma_{K^{\circ}}$. 

We assume that $\gamma_{K^{\circ}}$ is strictly convex. This is equivalent
to requiring $K^{\circ}$ to be strictly convex, i.e. its boundary
does not contain any line segment. The latter happens, for example,
when $K$ is strictly convex and its boundary is $C^{1}$ (see \citet{MR2336304}).
As an example, let $K=\{x\,\mid\,\gamma_{q}(x)\le1\}$ for $q>1$,
where $\gamma_{q}(x)=|x|_{q}:=(|x_{1}|^{q}+|x_{2}|^{q})^{1/q}$ is
the $q\hspace{1bp}$-norm of $x=(x_{1},x_{2})$. Then $\gamma_{K^{\circ}}=\gamma_{p}$
is strictly convex, where $p=\frac{q}{q-1}$ is the dual exponent
to $q$.

Let 
\[
I(v):=\int_{U}\frac{1}{2}|Dv|^{2}-\eta v\, dx,
\]
where $\eta>0$. Let $u$ be the minimizer of $I$ over 
\[
W_{K}:=\{v\in H_{0}^{1}(U)\,\mid\,\gamma_{K}(Dv)\leq1\textrm{ a.e. }\}.
\]
It can be shown that $u$ is also the minimizer of $I$ over 
\[
\{v\in H_{0}^{1}(U)\,\mid\, v(x)\leq d_{K^{\circ}}(x,\partial U)\textrm{ a.e. }\}.
\]
For the proof see \citet{MR0346345}, \citet{MR1797872} and \citet{MR1}.

When $K$ is the unit disk, and therefore $\gamma_{K},\gamma_{K^{\circ}}$
are both the Euclidean norm, the above problem is the famous elastic-plastic
torsion problem. The regularity of the free boundary of elastic-plastic
torsion problem is studied by \citet{MR0412940,MR0521411}, \citet{MR534111},
\citet{MR552267}, and \citet*{MR563207}. Their work is explained
by \citet{MR679313}. 

In this paper, we extend their results to the more general problem
explained above. A motivation for our study was to fill the gap between
the known regularity results mentioned above and the still open question
of regularity of the minimizer of some convex functionals subject
to gradient constraints arising in random surfaces. To learn about
the latter see the work of \citet{MR2605868}.

In order to study the free boundary, we generalize the notion of ridge,
by replacing the Euclidean norm by other norms. We also consider the
singularities of the distance function (in the new norm) to the boundary
of a domain. These notions have been considered and applied before
by \citet{MR2094267}, and \citet{MR2336304}. In particular the work
of \citet{MR2336304} has considerable intersection with ours. The
difference between our work and theirs lies in that we allow less
regular domains, and consider norms that are less restricted in some
aspects than what they consider.

\section{The Ridge%
\footnote{In this section we do not need any assumptions about $\partial U$.%
}}

First, we start by generalizing the notion of ridge.
\begin{defn}
The \textbf{$K$-ridge} of $U$ is the set of all points $x\in U$
where 
\[
d_{K}(x):=d_{K}(x,\partial U)
\]
is not $C^{1,1}$ in any neighborhood $V$ of $x$. We denote it by
\[
R_{K}.
\]
\end{defn}
\begin{lem}
Suppose $\gamma_{K}$ is strictly convex. If $d_{K}(x)=\gamma_{K}(x-y)=\gamma_{K}(x-z)$
for two different points $y,z$ on $\partial U$, then $d_{K}$ is
not differentiable at $x$.\end{lem}
\begin{proof}
Along the segment $\overline{xy}$ (and similarly $\overline{xz}$)
we have 
\begin{eqnarray}
d_{K}\big(x+\frac{t}{\gamma_{K}(x-y)}(y-x)\big) &  & =\gamma_{K}\big(x+\frac{t}{\gamma_{K}(x-y)}(y-x)-y\big)\nonumber \\
 &  & =\gamma_{K}(x-y)-t.
\end{eqnarray}
Now suppose to the contrary that $d_{K}$ is differentiable at $x$,
then we have 
\begin{equation}
Dd_{K}(x)\cdot\frac{y-x}{\gamma_{K}(x-y)}=-1=Dd_{K}(x)\cdot\frac{z-x}{\gamma_{K}(x-z)}.
\end{equation}
We know that  $\gamma_{K^{\circ}}(Dd_{K}(x))=1$. Thus, by strict
convexity of $\gamma_{K}$, there is at most one direction $e$ with
$\gamma_{K}(e)=1$ along which 
\[
D_{e}d_{K}(x)=Dd_{K}(x)\cdot e=-1.
\]
Therefore $d_{K}$ can not be differentiable at $x$.\end{proof}
\begin{defn}
The subset of the $K$-ridge consisting of the points with more than
one $d_{K}$-closest point on $\partial U$, is denoted by $R_{K,0}$.
\end{defn}
Now we return to our minimization problem. We assume that the minimizer,
$u$, is in $C^{1}(U)$. See \citet{MR1} for the proof of a stronger
regularity assuming some extra restrictions on $\gamma_{K^{\circ}}$.
\begin{lem}
We have 
\[
\{x\in U\,\mid\, u(x)=d_{K^{\circ}}(x)\}=\{x\in U\,\mid\,\gamma_{K}(Du(x))=1\},
\]
and 
\[
\{x\in U\,\mid\, u(x)<d_{K^{\circ}}(x)\}=\{x\in U\,\mid\,\gamma_{K}(Du(x))<1\}.
\]
\end{lem}
\begin{proof}
First suppose $u(x)=d_{K^{\circ}}(x)=\gamma_{K^{\circ}}(x-y)$ for
some $y\in\partial U$. Then by Lemma \ref{segment is plastic} we
have $D_{l}u(x)=-1$, where $l$ is the direction of $\overline{xy}$
segment (with $\gamma_{K^{\circ}}(l)=1$). Therefore $\gamma_{K}(Du(x))$
can not be less than $1$.

Next suppose $D_{l}u(x)=1$ for some $l$ with $\gamma_{K^{\circ}}(l)=1$,
and $u(x)<d_{K^{\circ}}(x)$. We know that $Du$ is harmonic in the
component of the open set $u<d_{K^{\circ}}$ containing $x$. As $|D_{l}u|\leq1$
on  this open set, the strong maximum principle implies that $D_{l}u$
is constant $1$ there. Now consider the line containing $x$ in the
$l$ direction and suppose it intersects the boundary of the open
set in $y$. If $y\in\partial U$, then moving along the line we see
that $u=d_{K^{\circ}}$,%
\footnote{Actually, along this segment we have $u(\cdot)=d_{K^{\circ}}(\cdot,y)\ge d_{K^{\circ}}(\cdot)$.%
} contradicting the choice of $x$. If $y\in U$, then we must have
$u(y)=d_{K^{\circ}}(y)$. Also since $D_{l}u(y)=1$, $l$ must be
the direction that connects $y$ to (one of) its $d_{K^{\circ}}$-closest
point on $\partial U$. Again we can see that along the line we have
$u=d_{K^{\circ}}$, which is a contradiction.\end{proof}
\begin{defn}
The set 
\[
\{x\in U\,\mid\, u(x)=d_{K^{\circ}}(x)\}=\{x\in U\,\mid\,\gamma_{K}(Du(x))=1\}
\]
is called the \textbf{plastic} region and is denoted by $P$. The
set 
\[
\{x\in U\,\mid\, u(x)<d_{K^{\circ}}(x)\}=\{x\in U\,\mid\,\gamma_{K}(Du(x))<1\}
\]
is called the \textbf{elastic} region and is denoted by $E$.\end{defn}
\begin{lem}
\label{segment is plastic}Suppose $x\in P$, and $y\in\partial U$
satisfies $u(x)=\gamma_{K^{\circ}}(x-y)$. Then the segment $\overline{xy}$
is entirely plastic.%
\footnote{This segment is obviously inside $U$.%
}\end{lem}
\begin{proof}
Consider $v=u-d_{K^{\circ}}$. Then along the segment $\overline{xy}$
we have $D_{l}d_{K^{\circ}}=-1$, where $l$ is the direction of the
segment with $\gamma_{K^{\circ}}(l)=1$. Thus $D_{l}v\geq0$ along
the segment, as $|D_{l}u|\leq1$. Now since $v(x)=v(y)=0$, we have
$v\equiv0$. Therefore $u=d_{K^{\circ}}$ along the segment as desired.\end{proof}
\begin{thm}
We have 
\[
R_{K^{\circ},0}\subset E.
\]
\end{thm}
\begin{proof}
Suppose to the contrary that $x\in R_{K^{\circ},0}\cap P$. Then there
are at least two points $y,z\in\partial U$ such that 
\[
d_{K^{\circ}}(x)=\gamma_{K^{\circ}}(x-y)=\gamma_{K^{\circ}}(x-z).
\]
Now by the above lemma we must have $\overline{xy},\overline{xz}\subset P$.
In other words $u=d_{K^{\circ}}$ on both $\overline{xy}$ and $\overline{xz}$.
This implies that 
\[
Du(x)\cdot\frac{y-x}{\gamma_{K^{\circ}}(y-x)}=-1=Du(x)\cdot\frac{z-x}{\gamma_{K^{\circ}}(z-x)},
\]
which is a contradiction.
\end{proof}

\section{The Case of $p\hspace{2bp}$-norms}

From now on we consider $\gamma_{K^{\circ}}$ to be the $p\hspace{1bp}$-norm,
$\gamma_{p}$, with $p\geq2$, where 
\[
\gamma_{p}((x_{1},x_{2})):=(|x_{1}|^{p}+|x_{2}|^{p})^{\frac{1}{p}}.
\]
We denote the corresponding ridge by $R_{p}$ and call it the $p\hspace{1bp}$-ridge.
In this case, the minimizer is in $C_{\textrm{loc}}^{1,1}(U)$ (see
\citet{MR1}).
\begin{lem}
The  $R_{p,0}$ of a disk, i.e. the set of points inside the disk
which have more than one $p$-closest point on its boundary, consists
of the union of its diagonals parallel to the coordinate axes.\end{lem}
\begin{proof}
It is easy to see that we can inscribe a $p\hspace{1bp}$-circle inside
the disk centered on the prescribed diagonals and touching the boundary
of the disk in at least two points, so the $R_{p,0}$ contains them.
Using Lagrange multipliers we can see that centered at any other point
in the disk, we can inscribe a $p\hspace{1bp}$-circle touching its
boundary in exactly one point.
\end{proof}
Next let us take the circle $(x-a)^{2}+(y-b)^{2}=r^{2}$ and assume
that at the point $(x_{0},y_{0})$ on it, the inscribed $p\hspace{1bp}$-circle
$|x-c|^{p}+|y-d|^{p}=s^{p}$ is tangent to the circle. We are interested
in the direction $(c-x_{0},d-y_{0})$. This is the direction of the
segment starting at $(x_{0},y_{0})$ along which, $(x_{0},y_{0})$
is the $p\hspace{1bp}$-closest point on the circle to its points
(for nearby points).

At $(x_{0},y_{0})$ the tangents, hence the normals, of the circle
and the $p\hspace{1bp}$-circle are parallel. Therefore we have 
\[
2(a-x_{0},b-y_{0})=\lambda p((c-x_{0})|c-x_{0}|^{p-2},(d-y_{0})|d-y_{0}|^{p-2}).
\]
Thus the direction of $(c-x_{0},d-y_{0})$ is (parallel to) 
\begin{equation}
(f_{p}(a-x_{0}),f_{p}(b-y_{0})),
\end{equation}
where $f_{p}$ is the inverse of the one to one map $t\mapsto t|t|^{p-2}=\textrm{sgn}(t)|t|^{p-1}$.
It is easy to see that $f_{p}$ is odd, and $|f_{p}(t)|=|t|^{\frac{1}{p-1}}$.
Also for nonzero $t$ we have $f_{p}(t)=t|t|^{\frac{2-p}{p-1}}$,
and $f_{p}'(t)=\frac{1}{p-1}|t|^{\frac{2-p}{p-1}}$. The following
lemma is easy to prove.
\begin{lem}
\label{lemma insc p-circ}If $x_{0}\neq a\,,\, y_{0}\neq b$, then
there is an inscribed $p$-circle inside the circle which is touching
the circle only at $(x_{0},y_{0})$. 
\end{lem}
Next, we introduce a notion of curvature for curves in the plane.
\begin{defn}
The \textbf{$p\hspace{1bp}$-curvature} of a curve $(x(t),y(t))$
in the plane is 
\begin{equation}
\kappa_{p}:=\frac{x'y''-y'x''}{(p-1)|x'|^{\frac{p-2}{p-1}}|y'|^{\frac{p-2}{p-1}}(|x'|^{\frac{p}{p-1}}+|y'|^{\frac{p}{p-1}})^{\frac{p+1}{p}}}.
\end{equation}

\end{defn}
It is easy to see that $\kappa_{p}$ does not change under the reparametrization
of the curve, hence it is an intrinsic quantity.

Note that the $p\hspace{1bp}$-curvature is not defined at the points
where the tangent (hence the normal) to the curve is parallel to the
coordinate axes. It also has the same sign as the ordinary curvature.
It is also easy to see that the $p\hspace{1bp}$-curvature of a $p\hspace{1bp}$-circle
is the inverse of its $p\hspace{1bp}$-radius at all points where
it is defined.

One notable difference between the usual curvature and the $p\hspace{1bp}$-curvature
for $p>2$ is that $\kappa_{p}$ does not appear in the first variation
of the $p\hspace{1bp}$-length.
\begin{lem}
Suppose $x_{i}\in U$ converge to $x\in U$, and $y\in\partial U$
is the unique $p$-closest point to $x$. If $y_{i}\in\partial U$
is a (not necessarily unique) $p$-closest point to $x_{i}$, then
$y_{i}$ converges to $y$.%
\footnote{This lemma does not need any assumptions about $\partial U$.%
}\end{lem}
\begin{proof}
Suppose that this does not happen. Then a subsequence of $y_{i}$,
which we denote it again by $y_{i}$, will remain outside an open
ball around $y$. Take this ball to be small enough so that $x$ lies
outside of it. Now consider a $p\hspace{1bp}$-circle around $x$
that touches $\partial U$ only at $y$ (The $p\hspace{1bp}$-radius
of the $p\hspace{1bp}$-circle is obviously $d_{p}(x)$). Consider
the subsets of this $p\hspace{1bp}$-circle and $\partial U$ that
lie outside the above ball. These subsets are compact sets that do
not intersect. Therefore they have a positive $p\hspace{1bp}$-distance,
$\delta$, of each other. Let $\epsilon<\delta/2$ be small enough
so that the $p\hspace{1bp}$-circle of $p\hspace{1bp}$-radius $\epsilon$
around $x$ is inside the above $p\hspace{1bp}$-circle, and does
not touch the ball around $y$. As $x_{i}$ approach $x$, they will
be inside the $p\hspace{1bp}$-circle with $p\hspace{1bp}$-radius
$\epsilon$. Thus 
\[
d_{p}(x)+\delta\le\gamma_{p}(x-y_{i})\le\gamma_{p}(x-x_{i})+\gamma_{p}(x_{i}-y_{i})<\epsilon+\gamma_{p}(x_{i}-y_{i}).
\]
Hence 
\[
d_{p}(x)+\delta-\epsilon<d_{p}(x_{i}).
\]
But, as $d_{p}$ is continuous, for small enough $\epsilon$ we must
have 
\[
d_{p}(x_{i})<d_{p}(x)+\delta/2,
\]
which is a contradiction.
\end{proof}
We now impose some extra restrictions on $\partial U$, in order to
study the regularity of $d_{p}$. We denote the inward normal to $\partial U$
by $\nu$.
\begin{assumption}
\label{assu: 1}We assume that at the points where the normal to $\partial U$
is parallel to one of the coordinate axes, the curvature of $\partial U$
is small. In the sense that, if we have $(t,b(t))$ as a nondegenerate
$C^{m,\alpha}$ $(m\geq2\,,\;0\leq\alpha<1)$ parametrization of $\partial U$
around $y_{0}:=(0,b(0))$, and $b'(0)=0$; then we assume $b'$ goes
fast enough to $0$ so that $b'(t)=c(t)|c(t)|^{p-2}$, where $c(0)=0$,
and $c$ is $C^{m-1,\alpha}$. 

Also we require $c'(0)$ to be less than $1$, and small enough so
that $1-c'(0)d_{p}(\cdot)$ does not vanish at the points inside $U$
that have $y_{0}$ as the only $p$-closest point on $\partial U$.%
\footnote{The requirement that $c'(0)<1$ is only used to show that there is
a $p$-circle in $U$ that touches $\partial U$ only at $y_{0}$,
but after examining this more carefully, it became apparent that this
restriction is superfluous.%
}
\end{assumption}
We will call these points the degenerate points of Assumption \ref{assu: 1}.
\begin{rem}
These assumptions will imply that there is a $p\hspace{1bp}$-circle
inside $U$ touching $\partial U$ only at $y_{0}$. \end{rem}
\begin{thm}
\label{thm d_p smooth}Suppose $\partial U$ is $C^{m,\alpha}$ $(m\geq2\,,\;0\leq\alpha<1)$
or analytic, and satisfies Assumption \ref{assu: 1}. Then outside
$R_{p,0}$, $d_{p}=d_{p}(\cdot,\partial U)$ is $C^{m,\alpha}$ or
analytic, except at the points with 
\begin{equation}
\kappa_{p}(y(x))d_{p}(x)=1.
\end{equation}
Where $y(x)$ is the $p$-closest point on $\partial U$ to $x$.\end{thm}
\begin{proof}
Suppose $y\in\partial U$ and $\nu(y)=(a(y),b(y))$ is not parallel
to the coordinate axes. Then as $\nu$ is smooth, it is not parallel
to the coordinate axes in a neighborhood of $y$. Also as $\partial U$
is smooth, there is a circle inside $U$ that is tangent to $\partial U$
only at $y$. Then $\nu(y)$ is also normal to the circle at $y$.
By Lemma \ref{lemma insc p-circ} there is a $p\hspace{1bp}$-circle
inside the circle and tangent to it (only) at $y$. Then we know that
the points on the segment joining $y$ to the center of the $p\hspace{1bp}$-circle
have $y$ as the $p\hspace{1bp}$-closest point to them on $\partial U$.
Also 
\begin{equation}
\mu(y)=\frac{1}{(|a|^{\frac{p}{p-1}}+|b|^{\frac{p}{p-1}})^{\frac{1}{p}}}(f_{p}(a),f_{p}(b))
\end{equation}
is the direction along which $y$ is the $p\hspace{1bp}$-closest
point. Note that by the above assumption on the $\partial U$, this
formula also gives the correct direction at points where one of $a$
or $b$ is zero. Also note that as $f_{p}(t)$ has the same sign as
$t$, $\mu$ is pointing inward $U$. 
\begin{defn}
We call $\mu$ the \textbf{inward $p\hspace{1bp}$-normal} to $\partial U$
at $y$.
\end{defn}
If two points have the same point on $\partial U$ as their $p\hspace{1bp}$-closest
point, then a $p\hspace{1bp}$-circle centered at them will be tangent
to $\partial U$ at that point. As the inward normal to a point on
a $p\hspace{1bp}$-circle uniquely determines the point, the two points
must be collinear with the point on the boundary. Hence the above
normal bundle $\mu$ to $\partial U$ gives at each point the only
direction along which the point is the $p\hspace{1bp}$-closest one.

Let $x_{0}$ be a point in $U-R_{p,0}$ and $y_{0}=y(x_{0})$.  Let
$y(t)=(a(t),b(t))$ for $|t|<L$, be a nondegenerate $C^{m,\alpha}$
or analytic parametrization of a segment of $\partial U$ with $y(0)=y_{0}$.
Also suppose that $(a+\epsilon,b)$ is in $U$ for small positive
$\epsilon$.  Now consider the map 
\[
F\,:\,(t,d)\mapsto y(t)+\mu(y(t))d
\]
from the open set $(-L,L)\times(0,\infty)$ into $\mathbb{R}^{2}$.
We have $F(0,d_{p}(x_{0}))=x_{0}$. We wish to compute $DF$ around
this point. First we deal with case that none of the $a',b'$ vanish
at $0$ (and hence around $0$). We have $\nu(t)=(-b'(t),a'(t))$
and 
\begin{equation}
\mu(t)=\frac{(-f_{p}(b'(t)),f_{p}(a'(t)))}{(|a'(t)|^{\frac{p}{p-1}}+|b'(t)|^{\frac{p}{p-1}})^{\frac{1}{p}}}.
\end{equation}
So we have 
\begin{eqnarray}
 & \mu'(t) & =\frac{(-(a')^{2}b''+a'b'a'',(b')^{2}a''-a'b'b'')}{(p-1)|a'|^{\frac{p-2}{p-1}}|b'|^{\frac{p-2}{p-1}}(|a'|^{\frac{p}{p-1}}+|b'|^{\frac{p}{p-1}})^{\frac{p+1}{p}}}\\
 &  & =-\kappa_{p}(a',b').\nonumber 
\end{eqnarray}
Hence 
\begin{equation}
DF=\left(\begin{array}{ccc}
a'-\kappa_{p}a'd &  & b'-\kappa_{p}b'd\\
\\
\frac{-b'|b'|^{\frac{2-p}{p-1}}}{(|a'|^{\frac{p}{p-1}}+|b'|^{\frac{p}{p-1}})^{\frac{1}{p}}} &  & \frac{a'|a'|^{\frac{2-p}{p-1}}}{(|a'|^{\frac{p}{p-1}}+|b'|^{\frac{p}{p-1}})^{\frac{1}{p}}}
\end{array}\right).
\end{equation}
Thus 
\begin{equation}
\det DF=(|a'|^{\frac{p}{p-1}}+|b'|^{\frac{p}{p-1}})^{\frac{p-1}{p}}(1-\kappa_{p}d)=\gamma_{q}((a',b'))(1-\kappa_{p}d),
\end{equation}
where $q$ is the dual exponent to $p$. This implies that $F$ is
$C^{m-1,\alpha}$ around $(0,d_{p}(x_{0}))$ with a $C^{m-1,\alpha}$
inverse, since $\kappa_{p}(y_{0})d_{p}(x_{0})\neq1$.

Next suppose one of the $a',b'$ vanishes. Thus the normal to $\partial U$
is parallel to the coordinate axes. We assume that $b'(0)=0$. The
other case is similar. Note that by the assumption on $\partial U$
we still have 
\[
\mu(t)=\frac{(-f_{p}(b'(t)),f_{p}(a'(t)))}{(|a'(t)|^{\frac{p}{p-1}}+|b'(t)|^{\frac{p}{p-1}})^{\frac{1}{p}}}.
\]
The problem is that $f_{p}$ is not differentiable at $0$. To avoid
this problem, we use the assumption that $b'$ goes fast enough to
$0$. In other words, $b'(t)=c(t)|c(t)|^{p-2}$ where $c(0)=0$, and
$c$ is $C^{m-1,\alpha}$. Hence $f_{p}(b'(t))=c(t)$. We can also
assume that $a(t)=t$ with no loss of generality, due to the inverse
function theorem. Therefore 
\begin{equation}
\mu(t)=\frac{(-c,1)}{(1+|c|^{p})^{\frac{1}{p}}},
\end{equation}
and 
\begin{eqnarray}
 & \mu'(t) & =\frac{(-c',0)}{(1+|c|^{p})^{\frac{1}{p}}}-\frac{c|c|^{p-2}c'(-c,1)}{(1+|c|^{p})^{\frac{p+1}{p}}}\nonumber \\
 &  & =\frac{-c'}{(1+|c|^{p})^{\frac{p+1}{p}}}(a',b').
\end{eqnarray}
(Note that at $t\neq0$ we have $\kappa_{p}=\frac{c'}{(1+|c|^{p})^{\frac{p+1}{p}}}$,%
\footnote{This equality only holds at the points where $c\ne0$. But if we define
$\kappa_{p}$ to be $c'$ at the points where $c=0$, $\kappa_{p}$
becomes continuous.%
} and this expression is continuous at $0$ with the value $c'(0)$
there.) Hence 
\begin{equation}
DF=\left(\begin{array}{ccc}
1-\frac{c'}{(1+|c|^{p})^{\frac{p+1}{p}}}d &  & c|c|^{p-2}-\frac{c'c|c|^{p-2}}{(1+|c|^{p})^{\frac{p+1}{p}}}d\\
\\
\frac{-c}{(1+|c|^{p})^{\frac{1}{p}}} &  & \frac{1}{(1+|c|^{p})^{\frac{1}{p}}}
\end{array}\right).
\end{equation}
Thus 
\begin{equation}
\det DF=(1+|c|^{p})^{\frac{p-1}{p}}(1-\frac{c'}{(1+|c|^{p})^{\frac{p+1}{p}}}d).
\end{equation}
Since we are only interested at $t=0$, 
\begin{equation}
\det DF(0,d)=1-c'(0)d.
\end{equation}
We assumed that $c'(0)$ is small enough so that inside $U$ this
determinant does not vanish. Therefore $F$ is $C^{m-1,\alpha}$ at
these points too, with a $C^{m-1,\alpha}$ inverse around them. 

Since $F\,:\,(t,d)\mapsto x$ is invertible in a neighborhood of $(0,d_{p}(x_{0}))$,
we have 
\begin{equation}
x=F(t(x),d(x))=y(t(x))+\mu(t(x))d(x).
\end{equation}
We also know that in general 
\[
x=y(x)+\mu(y(x))d_{p}(x).
\]
Note that $y(x)$ need not be unique for this formula to hold. If
we take $x$ close enough to $x_{0}$, then by continuity $y(x),d_{p}(x)$
will be close to $y(x_{0}),d_{p}(x_{0})$. And by invertibility of
$F$ we get 
\[
y(x)=y(t(x))\;,\; d_{p}(x)=d(x).
\]
As we showed that $x\mapsto(t,d)$ is locally $C^{m-1,\alpha}$, we
obtain that $d_{p}(x)$ and $y(x)$ are also locally $C^{m-1,\alpha}$.
Note that this also shows that all points around $x_{0}$ have a unique
$p\hspace{1bp}$-closest point around $y_{0}$, which by continuity
is the unique $p\hspace{1bp}$-closest point to them on $\partial U$.

 Next let us compute the first derivative of $d_{p}$ and $t$. Looking
at $DF$ we can compute 
\begin{equation}
DF^{-1}=\left(\begin{array}{ccc}
\frac{a'|a'|^{\frac{2-p}{p-1}}}{(1-\kappa_{p}d)(|a'|^{\frac{p}{p-1}}+|b'|^{\frac{p}{p-1}})} &  & \frac{-b'}{\gamma_{q}(a',b')}\\
\\
\frac{b'|b'|^{\frac{2-p}{p-1}}}{(1-\kappa_{p}d)(|a'|^{\frac{p}{p-1}}+|b'|^{\frac{p}{p-1}})} &  & \frac{a'}{\gamma_{q}(a',b')}
\end{array}\right).
\end{equation}
Which implies 
\begin{equation}
Dd_{p}(x)=(\frac{-b'}{\gamma_{q}((a',b'))},\frac{a'}{\gamma_{q}((a',b'))}),
\end{equation}
and 
\begin{equation}
Dt(x)=\frac{(a'|a'|^{\frac{2-p}{p-1}},b'|b'|^{\frac{2-p}{p-1}})}{(1-\kappa_{p}d)(|a'|^{\frac{p}{p-1}}+|b'|^{\frac{p}{p-1}})}.
\end{equation}
Specially note that as $a',b',t$ are $C^{m-1,\alpha}$, $d_{p}$
is $C^{m,\alpha}$. A similar conclusion can be made for the degenerate
points noting that 
\begin{equation}
DF^{-1}=\left(\begin{array}{ccc}
\frac{1}{(1+|c|^{p})(1-\frac{c'}{(1+|c|^{p})^{\frac{p+1}{p}}}d)} &  & \frac{-c|c|^{p-2}}{(1+|c|^{p})^{\frac{p-1}{p}}}\\
\\
\frac{c}{(1+|c|^{p})(1-\frac{c'}{(1+|c|^{p})^{\frac{p+1}{p}}}d)} &  & \frac{1}{(1+|c|^{p})^{\frac{p-1}{p}}}
\end{array}\right).
\end{equation}
\end{proof}
\begin{thm}
\label{laplace d_p}We have%
\footnote{Outside $R_{p}$.%
} 
\begin{equation}
\Delta d_{p}=\frac{-(p-1)\tau_{p}\kappa_{p}}{1-\kappa_{p}d_{p}},\label{eq: laplace d_p}
\end{equation}
where 
\begin{equation}
\tau_{p}=\frac{(|b'|^{\frac{2}{p-1}}+|a'|^{\frac{2}{p-1}})|a'b'|^{\frac{p-2}{p-1}}}{(|a'|^{\frac{p}{p-1}}+|b'|^{\frac{p}{p-1}})^{\frac{2p-2}{p}}}
\end{equation}
is a positive reparametrization invariant quantity along $\partial U$.
At the degenerate points of Assumption \ref{assu: 1} we have $\Delta d_{p}=0$.\end{thm}
\begin{proof}
We know that 
\[
Dd_{p}(x)=(\frac{-b'}{\gamma_{q}(a',b')},\frac{a'}{\gamma_{q}(a',b')})=\frac{(-b'(t),a'(t))}{(|a'|^{\frac{p}{p-1}}+|b'|^{\frac{p}{p-1}})^{\frac{p-1}{p}}}.
\]
Therefore 
\begin{eqnarray*}
 &  & D_{11}d_{p}=\\
 &  & =[\frac{-b''}{(|a'|^{\frac{p}{p-1}}+|b'|^{\frac{p}{p-1}})^{\frac{p-1}{p}}}+\frac{b'(a'|a'|^{\frac{2-p}{p-1}}a''+b'|b'|^{\frac{2-p}{p-1}}b'')}{(|a'|^{\frac{p}{p-1}}+|b'|^{\frac{p}{p-1}})^{\frac{2p-1}{p}}}]D_{1}t\\
 &  & =\frac{[-|a'|^{\frac{p}{p-1}}b''+b'a'|a'|^{\frac{2-p}{p-1}}a'']a'|a'|^{\frac{2-p}{p-1}}}{(1-\kappa_{p}d)(|a'|^{\frac{p}{p-1}}+|b'|^{\frac{p}{p-1}})^{\frac{3p-1}{p}}}.
\end{eqnarray*}
Similarly 
\begin{eqnarray*}
 &  & D_{22}d_{p}=\\
 &  & =[\frac{a''}{(|a'|^{\frac{p}{p-1}}+|b'|^{\frac{p}{p-1}})^{\frac{p-1}{p}}}-\frac{a'(a'|a'|^{\frac{2-p}{p-1}}a''+b'|b'|^{\frac{2-p}{p-1}}b'')}{(|a'|^{\frac{p}{p-1}}+|b'|^{\frac{p}{p-1}})^{\frac{2p-1}{p}}}]D_{2}t\\
 &  & =\frac{[|b'|^{\frac{p}{p-1}}a''-b'a'|b'|^{\frac{2-p}{p-1}}b'']b'|b'|^{\frac{2-p}{p-1}}}{(1-\kappa_{p}d)(|a'|^{\frac{p}{p-1}}+|b'|^{\frac{p}{p-1}})^{\frac{3p-1}{p}}}.
\end{eqnarray*}
Hence
\begin{eqnarray*}
 &  & \Delta d_{p}=\\
 &  & =\frac{a''b'(|b'|^{\frac{2}{p-1}}+|a'|^{\frac{2}{p-1}})-b''a'(|b'|^{\frac{2}{p-1}}+|a'|^{\frac{2}{p-1}})}{(1-\kappa_{p}d)(|a'|^{\frac{p}{p-1}}+|b'|^{\frac{p}{p-1}})^{\frac{3p-1}{p}}}\\
 &  & =\frac{-(p-1)\kappa_{p}}{1-\kappa_{p}d}\frac{(|b'|^{\frac{2}{p-1}}+|a'|^{\frac{2}{p-1}})|a'b'|^{\frac{p-2}{p-1}}}{(|a'|^{\frac{p}{p-1}}+|b'|^{\frac{p}{p-1}})^{\frac{2p-2}{p}}}\\
 &  & =\frac{-(p-1)\tau_{p}\kappa_{p}}{1-\kappa_{p}d_{p}}.
\end{eqnarray*}

For the degenerate case, we just need to note that by our assumptions
$\Delta d_{p}$ is continuous. Thus if we rewrite the above formula
in this case (for $t\neq0$) we get 
\begin{equation}
\Delta d_{p}=\frac{-(p-1)\tau_{p}\kappa_{p}}{1-\kappa_{p}d}=\frac{-(p-1)(1+c^{2})|c|^{p-2}c'}{(1+|c|^{p})^{\frac{3p-1}{p}}}\frac{1}{1-\frac{c'd_{p}}{(1+|c|^{p})^{\frac{p+1}{p}}}}.
\end{equation}
And as $t\rightarrow0$ this expression goes to $0$. Hence $\Delta d_{p}=0$
in this case. \end{proof}
\begin{thm}
The $p$-ridge consists of $R_{p,0}$ and those points outside of
it at which 
\[
1-\kappa_{p}d_{p}=0.
\]
\end{thm}
\begin{proof}
First note that we assumed as before that $1-c'(0)d_{p}$ does not
vanish inside $U$.%
\footnote{If we do not assume this, we cannot carry over the proof of this theorem
to the case of degenerate points of Assumption \ref{assu: 1}, since
$\Delta d_{p}=0$ there.%
} Thus we do not need to consider the degenerate case.

So far we showed that the $p\hspace{1bp}$-ridge contains $R_{p,0}$,
and every other point at which $1-\kappa_{p}d_{p}\neq0$ is not in
the $p\hspace{1bp}$-ridge. Now suppose $1-\kappa_{p}(y(x))d_{p}(x)=0$.
Then $\kappa_{p}(y)=\frac{1}{d_{p}(x)}>0$. We claim that $\Delta d_{p}$
blows up as we approach $x$ from a neighborhood of the line segment
$\overline{xy}$. The reason is that on this segment $1-\kappa_{p}d_{p}\neq0$
(hence also in a neighborhood of it by continuity). Thus we can apply
the above theorems to show that $d_{p}$ is at least $C^{2}$ on a
neighborhood of the segment, and its Laplacian is given by the formula
(\ref{eq: laplace d_p}). Therefore $d_{p}$ blows up, and can not
be $C^{1,1}$ in any neighborhood of $x$.\end{proof}
\begin{thm}
\label{thm: ridge is elastic}The $p$-ridge is elastic.\end{thm}
\begin{proof}
We already showed that $R_{p,0}$ is elastic. Therefore consider a
(nondegenerate) point $x$ not in $R_{p,0}$ at which $1-\kappa_{p}(y(x))d_{p}(x)=0$,
and suppose it is plastic. (Note that $\kappa_{p}(y(x))>0$.) Then
we know that the segment between $x$ and $y(x)$ is also plastic.
Hence along the segment $u=d_{p}$. As $u\leq d_{p}$ in general,
for $a$ in the segment we have 
\begin{equation}
\Delta_{h,\xi}^{2}u(a):=\frac{1}{h^{2}}(u(a-h\xi)+u(a+h\xi)-2u(a))\leq\Delta_{h,\xi}^{2}d_{p}(a).
\end{equation}
Where $\xi$ is the direction orthogonal to the direction of the segment,
$\zeta:=\frac{x-y}{|x-y|}$. Note that $d_{p}$ is linear along the
segment so $D_{\zeta\zeta}^{2}d_{p}$ vanishes there. Hence $\Delta d_{p}=D_{\xi\xi}^{2}d_{p}$
on the segment. Therefore 
\[
\underset{a\rightarrow x}{\lim}D_{\xi\xi}^{2}d_{p}(a)=\underset{a\rightarrow x}{\lim}\frac{-(p-1)\tau_{p}(y)\kappa_{p}(y)}{1-\kappa_{p}(y)d_{p}(a)}=-\infty.
\]
Thus for large positive $M$ we can choose $a$ close to $x$, and
$h(a)$ small enough, such that for $h\leq h(a)$ 
\[
\Delta_{h,\xi}^{2}u(a)\leq\Delta_{h,\xi}^{2}d_{p}(a)<-M.
\]
Fixing $a$ and taking the limit as $h\rightarrow0$ we get 
\[
D_{\xi\xi}^{2}u(a)\le-M.
\]
($D^{2}u$ must actually be interpreted as the Lipschitz constant
of $Du$.)%
\footnote{In fact $\Delta_{h,\xi}^{2}u(a)=\frac{1}{h}(D_{\xi}u(a+h_{1}\xi)-D_{\xi}u(a-h_{2}\xi))$
for some $0\le h_{i}\le h$. Thus we must have $|\Delta_{h,\xi}^{2}u(a)|\le2\|Du\|_{C^{0,1}}$.%
} But this contradicts $C^{1,1}$ regularity of $u$ around $x$.\end{proof}
\begin{lem}
\label{lemma 1-kd>0}$1-\kappa_{p}d_{p}>0$ on the free boundary.\end{lem}
\begin{proof}
Suppose $x_{0}$ is on the free boundary, and $y_{0}$ is the $p\hspace{1bp}$-closest
point on $\partial U$ to it. As the free boundary is part of the
plastic region, $1-\kappa_{p}(y_{0})d_{p}(x_{0})$ is nonzero. But
if it was negative we would have 
\[
1-\kappa_{p}(y_{0})d_{p}(x')=0
\]
for some $x'$ in the segment $\overline{x_{0}y_{0}}$. Since $d_{p}$
linearly grows along this segment and it is zero at $y_{0}$. Thus
$x'$ must be in the elastic region contradicting the fact that the
segment is plastic.

In the degenerate case of Assumption \ref{assu: 1} we need to replace
$1-\kappa_{p}d_{p}$ with $\det DF$. The proof is the same.
\end{proof}

\section{Domains with Corners}

We now relax some of the constraints on $\partial U$. We assume that
$\partial U$ consists of a finite number of (at least) $C^{2}$ (up
to the endpoints) disjoint arcs $S_{1},\cdots,S_{m}$. We denote by
$V_{i}$ the vertex $S_{i}\cap S_{i+1}$, and by $\alpha_{i}$ the
angle formed by $S_{i},S_{i+1}$.%
\footnote{We still assume that each component of $U$ is the interior of a simple
closed Jordan curve with finitely many holes.%
} 
\begin{defn}
We say the vertex $V_{i}$ is a \textbf{reentrant corner} if $\alpha_{i}\geq\pi$,
a \textbf{strict reentrant corner} if $\alpha_{i}>\pi$, and a \textbf{nonreentrant
corner} if $\alpha_{i}<\pi$.%
\footnote{When one of the normals to $\partial U$ at a nonreentrant corner
is parallel to one of the coordinate axes, we do not need to assume
Assumptions \ref{assu: 1} or \ref{assu: 2} hold there, except in
Section 7.

When $V_{i}$ is a reentrant corner and the normal to $S_{i}$ at
it is parallel to coordinate axes, either Assumption \ref{assu: 1}
must hold for $S_{i}$ at its endpoint $V_{i}$, or the modified version
of Assumption \ref{assu: 2} must hold, i.e. we assume that there
is no $p$-circle inside $U$ whose center is on the $p$-normal to
$S_{i}$ at $V_{i}$, and touches $\partial U$ only at $V_{i}$.%
}
\end{defn}
Let us look more closely at the arguments in the proof of regularity
of $d_{p}$. We see that as long as the point $y_{0}\in\partial U$
is not the $p\hspace{1bp}$-closest point on $\partial U$ to any
point in $U\backslash R_{p,0}$, we do not use any regularity assumption
about $\partial U$ at $y_{0}$. Therefore instead of Assumption \ref{assu: 1}
we can assume 
\begin{assumption}
\label{assu: 2}If $y_{0}\in\partial U$ is a point where the normal
to the boundary at it is parallel to one of the coordinate axes, then
there is no $p$-circle inside $U$ that touches its boundary only
at $y_{0}$.
\end{assumption}
In this case the $p\hspace{1bp}$-ridge can extend to $\partial U$
at $y_{0}$ in the direction of the normal to it. But it will still
remain inside the elastic region. (Note that by definition, the $p\hspace{1bp}$-ridge
is closed relative to $U$. Thus it remains outside a $U$-neighborhood
of the free boundary.) 

A sufficient condition for this assumption to hold, is that we can
parametrize $\partial U$ around these points as $(t,b(t))$ where
$b'(0)=0$ and $b''(0)>0$. This is true, for example, when $U$ is
a disk. 

It is also easy to see that as $p\hspace{1bp}$-circles have tangent
lines at every point, nonreentrant corners can not be the $p\hspace{1bp}$-closest
point on $\partial U$ to any point in $U$. 

We have the following generalization of Theorem \ref{thm d_p smooth}.
\begin{thm}
\label{thm: d_p smooth - corners}Suppose $\partial U$ is piecewise
$C^{m,\alpha}$ for $(m\geq2\,,\;0\leq\alpha<1)$ or piecewise analytic,
and satisfies Assumptions \ref{assu: 1} or \ref{assu: 2}. Then outside
$R_{p,0}$, $d_{p}=d_{p}(\cdot,\partial U)$ is $C^{m,\alpha}$ or
analytic around the points with $\kappa_{p}(y(x))d_{p}(x)\neq1$,
where $y(x)$ is the $p$-closest point on $\partial U$ to $x$ and
is not a reentrant corner. 

If $y(x)$ is a reentrant corner and the segment $\overline{xy}$
is between the inward $p$-normals to $y$, then $d_{p}$ is analytic
around $x$. And if the segment $\overline{xy}$ coincides with one
of the inward $p$-normals (or both of them if they coincide) and
$\kappa_{p}(y(x))d_{p}(x)\neq1$, where $\kappa_{p}$ is the $p$-curvature
of the corresponding boundary part, then $d_{p}$ is $C^{1,1}$ around
$x$. (But not $C^{2}$ in general.)\end{thm}
\begin{proof}
We just need to consider reentrant corners. Suppose $y_{0}=V_{1}$
is a strict reentrant corner and it is the $p\hspace{1bp}$-closest
point to $x_{0}$. Let $\mu_{1},\mu_{2}$ be the inward $p\hspace{1bp}$-normals
to respectively $S_{1},S_{2}$ at $y_{0}$. We assumed that if any
of the $\mu_{i}$'s is parallel to coordinate axes then one of the
Assumptions \ref{assu: 1} or \ref{assu: 2} holds.

Let $l$ be the tangent line at $y_{0}$ to $C_{0}$, where $C_{0}$
is the $p\hspace{1bp}$-circle around $x_{0}$ which passes through
$y_{0}$. $l$ must be between $l_{1},l_{2}$, the tangent lines to
$S_{1},S_{2}$ at $y_{0}$. Because the $p\hspace{1bp}$-circle is
inside the domain. If $l\neq l_{1},l_{2}$ then we claim that 
\[
d_{p}(x)=\gamma_{p}(x-y_{0})
\]
for $x$ close to $x_{0}$. Therefore $d_{p}$ is analytic there.
To prove the claim consider $C$, the $p\hspace{1bp}$-circle around
$x$ that passes through $y_{0}$. Then $C$ is in a small neighborhood
of $C_{0}$. As $C_{0}$ is inside $U$ except at $y_{0}$ and $U$
is open, we can see that $C$ is inside $U$ except possibly at a
neighborhood of $y_{0}$. But the tangent to $C$ at $y_{0}$ is close
to $l$ thus it is also strictly between $l_{1},l_{2}$. Therefore
$C$ is inside $U$ and touches $\partial U$ only at $y_{0}$.

Next we consider the case that $l=l_{1}$. Then $x_{0}$ is in the
direction of $\mu_{1}$. If $\mu_{1}$ is parallel to one of the coordinate
axes, Assumption \ref{assu: 1} must hold. We assume that $1-\kappa_{p}(y_{0})d_{p}(x_{0})\neq0$
where $\kappa_{p}(y_{0})$ is the $p\hspace{1bp}$-curvature of $S_{1}$
at $y_{0}$. Now let us extend $S_{1}$ in a smooth way so that $C_{0}$
still stays inside the new domain, and the new boundary stays below
$l$ outside of a small neighborhood of $x_{0}$. Then $d_{p}$ does
not change around $x_{0}$ on the side of $\mu_{1}$ that $S_{1}$
is. The reason is that the $p\hspace{1bp}$-circles around those points
which touch $S_{1}$ are either completely on the mentioned side of
$\mu_{1}$, or intersect the interior of the segment $\overline{x_{0}y_{0}}$.
The later $p\hspace{1bp}$-circles intersect $C_{0}$ close to $x_{0}$
on the side of $S_{1}$, since they also touch $S_{1}$ outside $C_{0}$
and close to $x_{0}$. Now note that two $p\hspace{1bp}$-circles
whose centers and $p\hspace{1bp}$-radii are close to each other will
intersect at exactly two almost antipodal points. Therefore the mentioned
$p\hspace{1bp}$-circles can not intersect the new boundary part,
which lies outside of $C_{0}$ on the opposite side of $S_{1}$. Here
we used the fact that two distinct $p\hspace{1bp}$-circle can intersect
at at most two points. To see this, note that two $p\hspace{1bp}$-circles
cannot intersect more than once on one side of the line that joins
their centers. Otherwise we can shrink or expand one of them so that
they become tangent at one point on that side of the line. But in
this situation the centers and the point of tangency must be collinear,
which is impossible.

Hence we can apply the regularity result for $d_{p}$ in the case
of smooth boundaries. Now consider a small disk around $x_{0}$ and
divide it by $\mu_{1}$. Then $d_{p}$ is smooth up to the boundary
of the half disk on the same side as $S_{1}$. On the other half,
$d_{p}$ is the $p\hspace{1bp}$-distance from $y_{0}$ hence it is
analytic. Let us compute $Dd_{p}(x_{0})$ from both sides. On the
side of $S_{1}$ we have 
\[
Dd_{p}(x_{0})=(\frac{-b'}{\gamma_{q}((a',b'))},\frac{a'}{\gamma_{q}((a',b'))})=\frac{\nu(y_{0})}{\gamma_{q}(\nu(y_{0}))},
\]
where $\nu(y_{0})$ is the inward normal to $S_{1}$ at $y_{0}$ (note
that this formula is also true in the degenerate case of Assumption
\ref{assu: 1}). On the other side we have 
\[
D_{i}d_{p}(x_{0})=D_{i}\gamma_{p}(x_{0}-y_{0})=\frac{((x_{0})_{i}-(y_{0})_{i})|(x_{0})_{i}-(y_{0})_{i}|^{p-2}}{\gamma_{p}(x_{0}-y_{0})^{p-1}}.
\]
But we have $\mu_{1}=\frac{x_{0}-y_{0}}{\gamma_{p}(x_{0}-y_{0})}$,
thus 
\begin{equation}
D_{i}d_{p}(x_{0})=\mu_{1i}|\mu_{1i}|^{p-2}.
\end{equation}
On the other hand 
\[
\mu_{1i}=\frac{f_{p}(\nu_{i}(y_{0}))}{\gamma_{q}(v(y_{0}))^{\frac{1}{p-1}}},
\]
where $f_{p}$ is the inverse of $t\mapsto t|t|^{p-2}$ as before.
Therefore the derivative of $d_{p}$ is the same from both sides and
it is $C^{1}$ around $x_{0}$. As $d_{p}$ is smooth on both sides
of $\mu_{1}$ up to $\mu_{1}$, we can say that it is in fact $C^{1,1}$
around $x_{0}$. But it is not in general $C^{2}$ there. To see this
we compute $\Delta d_{p}$ from both sides. On the side of $S_{1}$
we have 
\[
\Delta d_{p}(x_{0})=\frac{-(p-1)\tau_{p}(y_{0})\kappa_{p}(y_{0})}{1-\kappa_{p}(y_{0})d_{p}(x_{0})}.
\]
And on the other side 
\begin{eqnarray}
 & \Delta d_{p}(x_{0}) & =\sum D_{ii}^{2}\gamma_{p}(x_{0}-y_{0})\label{eq: laplace d_p nonre}\\
 &  & =(p-1)\sum\frac{|(x_{0})_{i}-(y_{0})_{i}|^{p-2}}{\gamma_{p}(x_{0}-y_{0})^{p-1}}-\frac{|(x_{0})_{i}-(y_{0})_{i}|^{2p-2}}{\gamma_{p}(x_{0}-y_{0})^{2p-1}}.\nonumber 
\end{eqnarray}
Now if, for example, $\kappa_{p}$ vanishes at $y_{0}$ and $(x_{0})_{i}\neq(y_{0})_{i}$
the two values will be different. 

If $\alpha_{1}=\pi$ (which means $\partial U$ is $C^{1}$ at $y_{0}$),
then we can repeat the same arguments. We can show that if a point
$x_{0}$ has $y_{0}$ as the only $p\hspace{1bp}$-closest point on
the boundary and satisfies $1-d_{p}(x_{0})(\kappa_{p})_{i}(y_{0})\neq0$,
where $(\kappa_{p})_{i}$ is the $p\hspace{1bp}$-curvature of $S_{i}$
(they can be different at $y_{0}$ as $\partial U$ is not necessarily
$C^{2}$ there), then $d_{p}$ is $C^{1,1}$ in a neighborhood of
$x_{0}$. But it is also not in general $C^{2}$.
\end{proof}
We can also show that if $1-\kappa_{p}(y_{0})d_{p}(x_{0})=0$ then
$x_{0}$ is in the ridge. The proof of this fact goes exactly as in
the case of smooth boundaries. The only modification is that we may
need to approach $x_{0}$ only from one side of the segment $\overline{x_{0}y_{0}}$.

\begin{thm}
\label{thm: ridge, corners}Suppose $\partial U$ is piecewise $C^{m,\alpha}$
for $(m\geq2\,,\;0\leq\alpha<1)$ or piecewise analytic, and satisfies
Assumptions \ref{assu: 1} or \ref{assu: 2}. Then the $p$-ridge
consists of $R_{p,0}$ and those points outside of it where $1-\kappa_{p}d_{p}=0$.
(The necessary adjustment must be made in the case of reentrant corners.)%
\footnote{\label{fn: ridge elastic}The $p$-ridge is also elastic in this case.
Although the proof of Theorem \ref{thm: ridge is elastic} must be
modified to show that if $x_{0}$ belongs to one of the $p$-normals
at a reentrant corner $y_{0}$ and $1-\kappa_{p}(y_{0})d_{p}(x_{0})=0$,
then $x_{0}$ is elastic. In this case we need to use the fact that
$Du=Dd_{p}$ on the interior of the segment $\overline{x_{0}y_{0}}$
when $x_{0}$ is plastic. Then moving in the direction $\xi$ orthogonal
to the segment, we must have $D_{\xi}u\le D_{\xi}d_{p}$ for some
points arbitrarily close to the segment. This easily gives a contradiction
with the boundedness of the $C^{0,1}$ norm of $Du$ around $x_{0}$.%
}
\end{thm}
Also it should be noted that at points where $d_{p}$ is smooth, we
can still compute $\Delta d_{p}$ as in Theorem \ref{laplace d_p}
or formula (\ref{eq: laplace d_p nonre}).

\section{Proof of the Regularity}

In this section we are going to modify the proof presented in \citet{MR679313}
in order to be applicable to our situation. The main difficulty for
doing so is that we are dealing with the inward $p\hspace{1bp}$-normal
instead of the inward normal. 

First let us assume that $\partial U$ is $C^{m,\alpha}$ for $(m\geq3\,,\;0<\alpha<1)$
or analytic, and satisfies Assumption \ref{assu: 1} at degenerate
points. We parametrize it by $y(t)$ for $0\leq t\leq L$. As before
we denote the inward $p\hspace{1bp}$-normal by $\mu=\mu(y(t))$. 
\begin{thm}
There exists a nonnegative function $\delta(t)$, such that the plastic
set is 
\begin{equation}
P=\{x\,\mid\, x=y(t)+s\mu(t)\,,\;0\leq s\leq\delta(t)\,,\;0\leq t\leq L\}.
\end{equation}
Moreover we have $d_{p}(y(t)+s\mu(t))=s$ for $0\leq s\leq\delta(t)$. \end{thm}
\begin{proof}
Remember that for $x\in U$, $y(x)$ is the $p\hspace{1bp}$-closest
point on $\partial U$ to $x$.We proved that if $x\in P$ then the
segment $\overline{xy(x)}$ is also in $P$. 

Consider the (connected) segment starting at the boundary point $y$
in the direction of the inward $p\hspace{1bp}$-normal at it, that
lies completely in $U$. Along this segment we can look at the supremum
of points that have $y$ as the only $p\hspace{1bp}$-closest point
to them on the boundary. This supremum can not be on $\partial U$,
as points close to that boundary point have $p\hspace{1bp}$-closest
point close to it. 

Now consider the supremum of points on the segment that have $y$
as the only $p\hspace{1bp}$-closest point, and are in $P$. This
supremum also belongs to $P$ as $P$ is closed. Also it has $y$
as the only $p\hspace{1bp}$-closest point on the boundary, as $P$
does not intersect the $p\hspace{1bp}$-ridge (specially $R_{p,0}$),
and the $p\hspace{1bp}$-closest point on the boundary changes continuously.
In addition, this supremum belongs to $U$ as it precedes the first
supremum.\end{proof}
\begin{rem}
Note that 
\begin{equation}
y(t)+(\delta(t)+\epsilon)\mu(t)\in E,
\end{equation}
where $0<\epsilon<\epsilon(t)$. The reason is that $x(t)=y(t)+\delta(t)\mu(t)$
is inside $U$ by the above proof.%
\footnote{When $\delta(t)=0$, we use the fact that a one-sided neighborhood
of $\partial U$ is inside $U$.%
}
\end{rem}
If we show that $\delta$ is continuous, then $\Gamma$, the free
boundary, is parametrized by $\delta(t)$ for $0\leq t\leq L$. 
\begin{thm}
\label{thm: free bdry jordan}$\delta$ is a continuous function of
$t$.\end{thm}
\begin{proof}
Suppose we want to show the continuity of $\delta$ at $t_{0}$ from
right. Let $t_{n}\searrow t_{0}$ and suppose that $\delta(t_{n})\rightarrow\delta(t_{0})+\epsilon$
where $\epsilon>0$. Then look at the points 
\[
x_{n}=y(t_{n})+(\delta(t_{0})+\epsilon/k)\mu(t_{n}),
\]
where $k$ is a big constant that makes $\tilde{x}=y(t_{0})+(\delta(t_{0})+\epsilon/k)\mu(t_{0})$
remain inside the elastic region $E$. Now for large $n$, $x_{n}$
is in the plastic region and we have $u(x_{n})=d_{p}(x_{n})$. But
by continuity of $u,d_{p}$ we must have $u(\tilde{x})=d_{p}(\tilde{x})$,
which is a contradiction. Therefore the limit points of the $\delta$
values of any sequence approaching $t_{0}$ must be in the interval
$[0,\delta(t_{0})]$. In particular if $\delta(t_{0})=0$ then $\delta$
is continuous at $t_{0}$.

Now we need to show that the limit points can not be less than $\delta(t_{0})$
either. First let us show that there is a sequence $t_{n}\searrow t_{0}$
such that $\delta(t_{n})\rightarrow\delta(t_{0})$. If this does not
happen then $\delta(t_{0})$ is outside the closed set of all limit
points of $\delta$-values of all sequences approaching $t_{0}$ from
right. Therefore an interval of the form $[\delta(t_{0})-\epsilon,\delta(t_{0})]$
is out of this set. This means that for $t$ close to $t_{0}$ we
have $\delta(t)<\delta(t_{0})-\epsilon$. So the open set 
\[
\{x\,\mid\, x=y(t)+s\mu(t)\,,\;\delta(t_{0})-\epsilon<s<\delta(t_{0})\,,\; t_{0}<t<t+\epsilon'\}
\]
is elastic and over it we have 
\[
-\Delta u=\eta.
\]
(This set is an open set by invariance of domain, as it is the image
of the locally injective continuous function $F(t,s)=y(t)+s\mu(t)$
introduced in the proof of Theorem \ref{thm d_p smooth}. We just
need to take $\epsilon,\epsilon'$ small enough.) Thus $u$ is analytic
there. Now on the segment 
\[
\{y(t_{0})+s\mu(t_{0})\;\mid\;\delta(t_{0})-\epsilon<s<\delta(t_{0})\}
\]
we have $u=d_{p}$ and $d_{p}$ is a linear function there. Also as
$u-d_{p}$ achieves its maximum there and it is $C^{1}$, we have
$Du=Dd_{p}$. But $Dd_{p}$ is constant along the segment and its
value depends only on $\mu(t_{0})$ there. Hence by uniqueness of
the solution of the Cauchy problem, locally, $u$ equals the linear
function whose derivative is $Dd_{p}$ plus a quadratic function of
the euclidean distance from the line containing the segment. (Note
that $-\Delta$ is elliptic so the segment is a noncharacteristic
surface for it.) As these two functions are analytic, they must be
equal on the component of the elastic region attached to the segment.
But this component contains the set 
\[
\{y(t_{0})+s\mu(t_{0})\;\mid\;\delta(t_{0})<s<\delta(t_{0})+\tilde{\epsilon}\}.
\]
Hence $u$ equals the linear function there, as the quadratic part
vanishes on it. However the linear function equals $d_{p}$ along
this segment too. Hence we get a contradiction.

Now suppose $s_{n}\searrow t_{0}$ and $\delta(s_{n})\rightarrow\delta(t_{0})-\epsilon$.
By taking subsequences we can assume $t_{n+1}<s_{n}<t_{n}$ where
$\delta(t_{n})\rightarrow\delta(t_{0})$. For $n$ large enough, $\delta(s_{n})$
is less than $\delta(t_{n}),\delta(t_{n+1})$. As we showed that $\delta$
is upper semicontinuous, it achieves its maximum over $[t_{n+1},t_{n}]$.
Let this maximum be $M_{n}$. The set $\{t\in(t_{n+1},t_{n})\,\mid\,\delta(t)<M_{n}\}$
must be open for the same reason. Let $(t'_{n+1},t'_{n})$ be the
component of it that contains $s_{n}$. Obviously the largest of $\delta(t'_{n+1}),\delta(t'_{n})$
equals $M_{n}$. Now take the set 
\begin{equation}
A_{n}=\{t'_{n+1}<t<t'_{n}\,,\, m_{n}<s<M_{n}\},
\end{equation}
where $m_{n}$ is the infimum of $\delta$ over $[t'_{n+1},t'_{n}]$.
We want $F(A_{n})$ to be an open set. Set 
\[
x_{0}=y(t_{0})+\delta(t_{0})\mu(t_{0}).
\]
Lemma \ref{lemma 1-kd>0} implies 
\[
1-\kappa_{p}(t_{0})d_{p}(x_{0})=1-\kappa_{p}(t_{0})\delta(t_{0})>0.
\]
(In the degenerate case of Assumption \ref{assu: 1} we need to replace
this expression with $\det DF$ at those points). We know that $F$
is injective on a neighborhood of $x_{0}$. We choose this neighborhood
small enough so that $1-\kappa_{p}d_{p}>0$ over it too. Let $n$
be large enough so that for $t_{n+1}<t<t_{n}$ the half-lines containing
inward $p\hspace{1bp}$-normals intersect this neighborhood. Then
as $1-\kappa_{p}d_{p}$ can change sign only once along these half-lines,
we conclude that it is positive on $F(A_{n})$ (Note that $M_{n}$
is close to $\delta(t_{0})$). Hence $F$ is at least $C^{1}$ on
$A_{n}$, and the determinant of its derivative never vanishes. This
implies that $F$ is injective on $A_{n}$. Otherwise the restriction
of $F$ to the segment connecting two points with the same $F$-value
would have a local extremum, which implies $DF$ has a zero eigenvalue
there, contradicting the assumptions. Therefore by invariance of domain
$F(A_{n})$ is open.

Now consider the function $v:=d_{p}-u$ over the open set 
\[
E_{n}:=E\cap F(A_{n}).
\]
We have $v>0$ and $\Delta v=\Delta d_{p}+\eta$ over it. First let
us show that $\Delta v$ has a positive infimum over $E_{n}$. Lemma
\ref{lemma f-Dphi>0} (which does not assume any regularity about
the free boundary) shows that $\Delta d_{p}+\eta>0$ on the free boundary.
Thus if we choose the neighborhood around $x_{0}$ small enough, then
as the tale of each segment (along the $p\hspace{1bp}$-normals) in
$F(A_{n})$ is in that neighborhood, $\Delta d_{p}+\eta$ will have
a positive infimum there. Also it is positive at the free boundary
point on each segment. We also know that 
\[
\Delta d_{p}+\eta=\eta-\frac{(p-1)\kappa_{p}\tau_{p}}{1-\kappa_{p}d_{p}}.
\]
Thus it can not change sign twice along these segments, as $d_{p}$
grows linearly along them. (In the degenerate case of Assumption \ref{assu: 1}
we have $\Delta d_{p}=0$ and this expression is certainly positive.)
Hence $\Delta d_{p}+\eta>0$ on $E_{n}$. And as it has a positive
infimum on the tale of segments, and by compactness a positive minimum
on their initial points on the free boundary, we can say it has a
positive infimum over $E_{n}$ (note that it is monotone along segments).
Therefore we have 
\[
\Delta v\geq\lambda>0
\]
on $E_{n}$ for some constant $\lambda$.%
\footnote{Note that $\lambda$ is independent of $n$.%
} Now let 
\begin{equation}
w(x):=v(x)-v(\bar{x})-\frac{\lambda}{2}|x-\bar{x}|^{2},
\end{equation}
where 
\[
\bar{x}=y(s_{n})+(\delta(s_{n})+\epsilon')\mu(s_{n})\in E_{n}.
\]
Then we have 
\[
\Delta w=\Delta v-\lambda\geq0,
\]
and $w(\bar{x})=0$. So by the maximum principle $\sup w\geq0$ and
is attained on $\partial E_{n}$. But on the free boundary and on
the segments 
\[
\{y(t)+s\mu(t)\,\mid\, t=t'_{n},t'_{n+1}\,,\, m_{n}\leq s\leq\delta(t)\},
\]
we have $v=0$ so $w<0$. Therefore $\sup w$ must be attained on
\[
B_{n}:=\{y(t)+M_{n}\mu(t)\,\mid\, t'_{n+1}\leq t\leq t'_{n}\}\cup\{y(t)+s\mu(t)\,\mid\, t=t'_{n},t'_{n+1}\,,\,\delta(t)<s\leq M_{n}\}.
\]
Hence for some $x'$ in this set we have 
\[
v(x')-v(\bar{x})-\frac{\lambda}{2}|x'-\bar{x}|^{2}=w(x')\geq0,
\]
or (noting that $M_{n}\ge\delta(t'_{n})\,,\,\delta(t'_{n+1})$ ) 
\begin{equation}
v(x')-v(\bar{x})\geq\frac{\lambda}{2}|x'-\bar{x}|^{2}\geq C\frac{\lambda}{2}\min\{|\delta(t'_{n+1})-\delta(s_{n})-\epsilon'|^{2}\,,\,|\delta(t'_{n})-\delta(s_{n})-\epsilon'|^{2}\},
\end{equation}
where $C$ depends on the maximum and minimum of the $C^{1}$ norm
of $F$ on a set of the form 
\[
\{t_{0}\leq t\leq t_{0}+\tilde{\epsilon}\,,\,0\leq s\leq M\},
\]
which can be shown similar to before that $F$ is $C^{1}$ over it
with a $C^{1}$ inverse. Letting $\epsilon'\rightarrow0$ we get 
\[
v(x')\geq C\frac{\lambda}{2}\min\{|\delta(t'_{n+1})-\delta(s_{n})|^{2}\,,\,|\delta(t'_{n})-\delta(s_{n})|^{2}\}.
\]
However $B_{n}$ intersects the free boundary at its end points. Let
the closest endpoint to $x'$ be $x''$. Then as $u$ is $C^{1,1}$
away from $\partial U$, we have 
\begin{equation}
v(x')\leq C'|x''-x'|+v(x'')\leq C'C[(t_{n}'-t_{n+1}')+(M_{n}-\delta(t'_{n}))+(M_{n}-\delta(t'_{n+1}))].
\end{equation}
Where $C'$ bounds $C^{1,1}$ norm of $v=d_{p}-u$ around $x_{0}$,
and $C$ is as before. Therefore we have 
\begin{eqnarray}
|\delta(t_{n})-\delta(s_{n})|^{2} &  & \leq\min\{|\delta(t'_{n+1})-\delta(s_{n})|^{2}\,,\,|\delta(t'_{n})-\delta(s_{n})|^{2}\}\\
 &  & \le\tilde{C}[(t_{n}'-t_{n+1}')+(M_{n}-\delta(t'_{n}))+(M_{n}-\delta(t'_{n+1}))]\underset{n\rightarrow\infty}{\rightarrow}0.\nonumber 
\end{eqnarray}
But this contradicts the assumption. Hence $\delta$ is continuous
at $t_{0}$.\end{proof}
\begin{rem}
If we allow $\partial U$ to be piecewise smooth and to satisfy Assumptions
\ref{assu: 1} or \ref{assu: 2} at degenerate points, then we can
still say that the free boundary is locally a continuous arc. To see
this take $x_{0}$ to be a point on the free boundary (inside $U$),
and let $y_{0}$ be the $p\hspace{1bp}$-closest point to it on the
boundary (we know that $x_{0}\notin R_{p,0}$). If $y_{0}$ is a regular
point or a degenerate point of Assumption \ref{assu: 1}, then around
it we can still represent the free boundary as the graph of a function
$\delta$ and the previous arguments apply. If $y_{0}$ is a reentrant
corner and $x_{0}$ is between the inward $p\hspace{1bp}$-normals
at $y_{0}$, then we can define a similar function $\delta$ (of the
angle between $\overline{x_{0}y_{0}}$ and one of the $p\hspace{1bp}$-normals)
whose graph is the free boundary and the analysis is similar to before.
(Note that in this case $\Delta d_{p}\geq0$ and $d_{p}$ is just
the $p\hspace{1bp}$-distance from the corner, so the analysis is
actually simpler.)

And finally if $\overline{x_{0}y_{0}}$ is in the direction of one
of the $p\hspace{1bp}$-normals, then we can prove the continuity
from each side separately. Although the above proof needs modifications
as we did not show that $1-\kappa_{p}d_{p}$ is necessarily nonzero
at these points. Therefore the ``open'' set in the proof is not
necessarily the image of an injective map.%
\footnote{The one-sided continuity from the region between the two $p$-normals
follows as before, if we are not in the degenerate case of Assumption
\ref{assu: 2}. Continuity from the other side follows similarly in
some cases, for example when $\kappa_{p}=0$ at the corner. %
} Nevertheless the analysis below does not apply at these points and
we can not get regularity at them. So we do not go into the details
of this case.
\end{rem}
Now we can apply the following theorem proved in \citet{MR679313}.
Note that the free boundary is locally a Jordan arc, since it is parametrized
by the continuous injective map 
\[
t\mapsto y(t)+\delta(t)\mu(t).
\]
Injectivity follows from the arguments in the beginning of this section
resulted in the definition of $\delta$, and the fact that the free
boundary is part of the plastic region. (Thus points of the free boundary
have a unique $p\hspace{1bp}$-closest point on $\partial U$ to them.)
The case of reentrant corners is similar.
\begin{thm}
\label{thm: free bdry smooth}Let $V$ be a simply connected domain
in $\mathbb{R}^{2}$ whose boundary $\partial V$ is a Jordan curve,
and let $\Gamma\subset\partial V\cap B_{R}(x_{0})$ be a Jordan arc
(for some $R>0$). Also assume that $u\in C^{1,1}(V\cup\Gamma)$,
$f\in C^{m,\alpha}(B_{R}(x_{0}))$, $\phi\in C^{m+2,\alpha}(B_{R}(x_{0}))$
where $m\geq0\,,\;0<\alpha<1$. And assume that $u$ satisfies 
\begin{equation}
\begin{array}{ccc}
\Delta u=f &  & \textrm{in }V\cap B_{R}(x_{0})\\
\\
u=\phi\,,\; Du=D\phi &  & \textrm{on }\Gamma\\
\\
f-\Delta\phi\neq0 &  & \textrm{in }B_{R}(x_{0}).
\end{array}
\end{equation}
Then $\Gamma$ has a $C^{m+1,\alpha}$ nondegenerate parametrization.
\end{thm}
Now we can take $x_{0}$ on the free boundary (inside $U$) to be
a point where $d_{p}$ is at least $C^{3,\alpha}$ for $0<\alpha<1$
around it. Take $V$ to be the subset of $E$, the elastic set, consisted
of points above the graph of $\delta$ around $x_{0}$. Note that
$V$ is simply connected as we can project any loop in it onto the
graph of $\delta$ and then shrink it to a point. We know that $u$
is in $C^{1,1}(\bar{V})$. If we show that $-\eta-\Delta d_{p}\neq0$
at $x_{0}$, then by continuity we can choose a small ball $B_{R}(x_{0})$
such that $-\eta-\Delta d_{p}\neq0$ everywhere on it, and the theorem
applies. We use the following lemma from \citet{MR679313} to show
this. Note that our problem is equivalent to 
\begin{equation}
\begin{array}{c}
-\Delta(-u)+\eta\geq0\\
\\
-u\geq-d_{p}\\
\\
(-\Delta(-u)+\eta)(-u+d_{p})=0.\\
\\
\end{array}
\end{equation}

\begin{lem}
\label{lemma f-Dphi>0}Suppose 
\[
\begin{array}{c}
-\Delta u+f\geq0\\
\\
u\geq\phi\\
\\
(-\Delta u+f)(u-\phi)=0
\end{array}
\]
in $U$, and $u\in C^{1,1}(U)$, $\phi\in C^{3}$. Then on the free
boundary, if $f-\Delta\phi$ and $D(f-\Delta\phi)$ do not vanish
simultaneously, we have $f-\Delta\phi>0$. (The free boundary is the
boundary of the set $\{u>\phi\}$ inside $U$.)
\end{lem}

Hence we only need to show that $\psi:=\eta+\Delta d_{p}>0$ along
the free boundary, by showing that $\psi,D\psi$ can not vanish simultaneously
there. In order to prove the latter fact, suppose at a nondegenerate
point $x_{0}$ we have 
\[
\psi=\eta+\Delta d_{p}=\eta-\frac{(p-1)\kappa_{p}\tau_{p}}{1-\kappa_{p}d_{p}}=0.
\]
As $\eta,\tau_{p}$ are positive, this implies that $\kappa_{p}(y(x_{0}))\neq0$.
(Note that as free boundary is plastic we have $1-\kappa_{p}d_{p}\neq0$
there.) Now we have 
\begin{equation}
D_{\mu}\psi=D_{\mu}\Delta d_{p}=\frac{-(p-1)\kappa_{p}^{2}\tau_{p}}{(1-\kappa_{p}d_{p})^{2}}D_{\mu}d_{p}=\frac{-(p-1)\kappa_{p}^{2}\tau_{p}}{(1-\kappa_{p}d_{p})^{2}}\neq0.
\end{equation}
Since $\kappa_{p},\tau_{p}$ do not change along $\mu=\mu(y(x_{0}))$
and $D_{\mu}d_{p}=1$. In case of degenerate points of Assumption
\ref{assu: 1}, we showed that there $\Delta d_{p}=0$ so $\psi>0$.
Finally if $y(x_{0})$ is a reentrant corner and $x_{0}$ is between
the inward $p\hspace{1bp}$-normals, formula (\ref{eq: laplace d_p nonre})
shows that $\Delta d_{p}\geq0$, hence $\psi>0$.

Now if we show that the free boundary has no cusps, we have proved
that it is a smooth curve. The following theorem proved in \citet{MR679313}
for $p=2$, but the proof works in this more general setting with
no change.
\begin{thm}
\label{thm: no cusps}The plastic set has positive density at each
point of the free boundary where $d_{p}$ is $C^{3,\alpha}$ around
it, for some $0<\alpha<1$.
\end{thm}
Therefore putting all these together we get
\begin{thm}
The free boundary is locally $C^{m,\alpha}$ ($m\geq2\,,\;0<\alpha<1$)
or analytic, if the part of $\partial U$ that parametrizes it is
$C^{m+1,\alpha}$ or analytic. Any open part of the free boundary
that has a reentrant corner as its $p$-closest point on the boundary,
is analytic.\end{thm}
\begin{proof}
Theorems \ref{thm: d_p smooth - corners} and \ref{thm: ridge, corners}
imply that $d_{p}$ is $C^{m+1,\alpha}$ around the free boundary.
Also by Theorem \ref{thm: free bdry jordan}, the free boundary is
locally a Jordan arc. Therefore we can apply Theorem \ref{thm: free bdry smooth}
with $\phi=d_{p}$ and conclude that locally, the free boundary has
a nondegenerate $C^{m,\alpha}$ parametrization. This together with
the fact that the free boundary has no cusps (Theorem \ref{thm: no cusps}),
give the result. The case of reentrant corners is similar, noting
that $d_{p}$ is analytic around them.
\end{proof}
Note that by this theorem the free boundary is smooth, except at the
finite number of points on it which have a reentrant corner as the
$p\hspace{1bp}$-closest point to them, and lie on the direction of
an inward $p\hspace{1bp}$-normal at that corner.

\section{The $p\hspace{2bp}$-bisector}

Now we focus on understanding the shape of the free boundary. Consider
the line $ax+by+c=0$ and the point $(x_{0},y_{0})$ in the plane.
We want to find the $p\hspace{1bp}$-distance of the point and the
line. The direction of the $p\hspace{1bp}$-normal to the line is
\[
(a|a|^{\frac{2-p}{p-1}},b|b|^{\frac{2-p}{p-1}}).
\]
The $p\hspace{1bp}$-closest point on the line to $(x_{0},y_{0})$,
is the intersection of the line and the line that passes through $(x_{0},y_{0})$
in the direction of the $p\hspace{1bp}$-normal. Let that point be
\[
(x_{0},y_{0})+t(a|a|^{\frac{2-p}{p-1}},b|b|^{\frac{2-p}{p-1}}).
\]
Then 
\[
a(x_{0}+ta|a|^{\frac{2-p}{p-1}})+b(y_{0}+tb|b|^{\frac{2-p}{p-1}})+c=0.
\]
Hence 
\begin{equation}
t=-\frac{ax_{0}+by_{0}+c}{|a|^{\frac{p}{p-1}}+|b|^{\frac{p}{p-1}}}.
\end{equation}
Therefore the $p\hspace{1bp}$-distance is 
\begin{eqnarray}
 & d_{p} & =|t|(|a|^{\frac{p}{p-1}}+|b|^{\frac{p}{p-1}})^{\frac{1}{p}}\nonumber \\
 &  & =\frac{|ax_{0}+by_{0}+c|}{(|a|^{\frac{p}{p-1}}+|b|^{\frac{p}{p-1}})^{\frac{p-1}{p}}}\nonumber \\
 &  & =\frac{|ax_{0}+by_{0}+c|}{(|a|^{q}+|b|^{q})^{\frac{1}{q}}},
\end{eqnarray}
where $q=\frac{p}{p-1}$ is the dual exponent to $p$. This also implies
that 
\begin{equation}
d_{p}=cd_{2}.\label{eq: d_p =000026 d_2}
\end{equation}
Where $c$ is some positive constant depending only on $p$ and the
line, but not on the point $(x_{0},y_{0})$.
\begin{defn}
The \textbf{$p\hspace{1bp}$-bisector} of an angle is the set of points
inside the angle that have equal $p\hspace{1bp}$-distance from each
side of the angle.
\end{defn}
It is easy to see from the above arguments that the $p\hspace{1bp}$-bisector
of an angle is a ray inside the angle emitting from its vertex.

\section{Nonreentrant Corners}
\begin{thm}
Suppose $w\in W^{1,\infty}(U)$, and $w\ge0$ on $\partial U$.%
\footnote{For this theorem, we assume that $\partial U$ is Lipschitz.%
} Also suppose that $-\Delta w\ge0$ in the weak sense, i.e. 
\[
\int_{U}Dw\cdot D\phi\, dx\ge0\qquad\qquad\textrm{ for all }\phi\in H_{0}^{1}(U)\textrm{ with }\phi\ge0\textrm{ a.e.}.
\]
Then $w\ge0$ on $U$.\end{thm}
\begin{proof}
Let $\phi=-w^{-}:=-\min\{w,0\}$. We have 
\[
-\int_{\{w<0\}}|Dw|^{2}\, dx=\int_{U}Dw\cdot D(-w^{-})\, dx\ge0.
\]
As $w$ is Lipschitz continuous, $\{w<0\}$ is an open set. But by
the above, $Dw=0$ a.e. on this open set. Also $w=0$ on $\partial\{w<0\}$.
Therefore $\{w<0\}$ must be empty, and $w\ge0$ everywhere on $U$.
\end{proof}
Now suppose
\begin{equation}
G:=\{(r,\theta)\,\mid\,0<r<r_{0}\,,\,-\alpha<\theta<\alpha\},
\end{equation}
and $\alpha<\pi/2$. 
\begin{lem}
\label{lemma u<rd_p}Suppose $u\in W^{1,\infty}(G)$ has Laplacian
bounded from below in the weak sense, i.e. for some $C_{1}>0$ we
have 
\[
\int_{G}Du\cdot D\phi\, dx\leq C_{1}\int_{G}\phi\, dx\qquad\qquad\textrm{ for all }\phi\in H_{0}^{1}(G)\textrm{ with }\phi\ge0\textrm{ a.e.}.
\]
Also suppose that $u$ vanishes on the straight sides of $G$. Then
there are positive constants $C,\nu$, such that for $r$ small enough
we have 
\begin{equation}
u(r,\theta)\le Cr^{\nu}d_{p}(r,\theta),\label{eq: |x|d}
\end{equation}
where $d_{p}$ is the $p$-distance from the straight sides of $G$.\end{lem}
\begin{proof}
First assume that $\alpha>\pi/4$ and consider the following function
on $G$ 
\begin{equation}
v:=r^{\frac{\pi}{2\alpha}}\cos(\frac{\pi}{2\alpha}\theta)+\frac{r^{2}}{2}(\sin^{2}\theta-\cos^{2}\theta\tan^{2}\alpha).
\end{equation}
It is easy to see that in $G$ 
\begin{equation}
\Delta v=1-\tan^{2}\alpha<0,
\end{equation}
and $v=0$ on $\theta=\pm\alpha$. If $r_{0}$ is small enough then
$v>0$ on $r=r_{0}$. Also $D_{\theta}v(r_{0},-\alpha)>0$ and $D_{-\theta}v(r_{0},\alpha)>0$
(because $\cos(\frac{\pi}{2\alpha}\theta)\geq0$ and $\frac{\pi}{2\alpha}<2$).
Hence for $C>0$ large enough 
\begin{equation}
\Delta(Cv-u)\le0\qquad\textrm{ in the weak sense}.
\end{equation}
And also $Cv-u>0$ on $r=r_{0}$. Since its inward derivative is positive
at $(r_{0},\pm\alpha)$, so it is positive around them. (Actually,
we only need to work with the Lipschitz constants.) Also on the remaining
part of $r=r_{0}$ it is positive, because the coefficient of the
dominant term of $v$ there, i.e. $\cos(\frac{\pi}{2\alpha}\theta)$,
has a positive minimum there. Also $Cv-u$ vanishes on $\theta=\pm\alpha$.
Therefore by the maximum principle we have 
\begin{equation}
u\leq Cv
\end{equation}
inside $G$. Hence 
\[
u(r,\theta)\leq2Cr^{\frac{\pi}{2\alpha}}\cos(\frac{\pi}{2\alpha}\theta)=2Cr^{\nu}l\frac{\cos(\frac{\pi}{2\alpha}\theta)}{\omega(\theta)}.
\]
Where $\nu=\frac{\pi}{2\alpha}-1>0$, and 
\[
l=r\omega(\theta)=\begin{cases}
r\sin(\theta+\alpha) & -\alpha<\theta<0\\
r\sin(\alpha-\theta) & 0\leq\theta<\alpha
\end{cases}
\]
is the euclidean distance of the point $(r,\theta)$ from the straight
sides of $G$. But the function $\frac{\cos(\frac{\pi}{2\alpha}\theta)}{\omega(\theta)}$
is bounded for $\theta\in(-\alpha,\alpha)$ as it is continuous and
has finite limits at $\pm\alpha$. Thus 
\[
u(r,\theta)\leq\tilde{C}r^{\nu}l.
\]
But by the formula (\ref{eq: d_p =000026 d_2}) we have 
\[
d_{p}(r,\theta)\geq cl.
\]
(Note that there are two sides and we have to take the minimum of
both the distance and the $p\hspace{1bp}$-distance to them, hence
the equality in the formula (\ref{eq: d_p =000026 d_2}) becomes an
inequality.) Hence for small $r$ we get 
\begin{equation}
u(r,\theta)\leq\frac{\tilde{C}}{c}r^{\nu}d_{p}(r,\theta).
\end{equation}

When $\alpha\le\pi/4$, we need to find an appropriate replacement
for $v$. Consider $w$, the minimizer of $I$ over $K(G)$. As proved
in the Remark \ref{remark another proof}, $w$ satisfies the bound
(\ref{eq: |x|d}). It also has negative Laplacian around the vertex
of $G$. Also, $w$ vanishes on the straight sides of $G$, and is
nonnegative inside it. Consider $\tilde{G}\subset G$, which consists
of all points with $r<r_{1}$, where $r_{1}$ is small enough so that
$\Delta w=-\eta$ on $\tilde{G}$. We can also choose $r_{1}$ such
that $w$ is positive on some point of $\partial\tilde{G}$. Since
otherwise $w=0$ on $\tilde{G}$ and can not satisfy the equation.
As $\Delta w<0$ inside $\tilde{G}$ and $w$ is nonnegative on $\partial\tilde{G}$,
strong maximum principle implies that $w>0$ inside $\tilde{G}$. 

Now let $\hat{G}\subset G$ be the set of points with $r<r_{1}-\epsilon$.
Then on the circular part of $\partial\hat{G}$ we have $w>0$. Also
the inward normal derivative of $w$ is positive on the endpoints
of this circular part. The reason is that it is nonnegative and if
it was zero we would have $Dw=0$ at that point (note that $w=0$
on the straight sides of $\partial G$). But this gives a contradiction
because the second derivative of $w$ along the inward normal must
be negative (since $\Delta w<0$ and the second derivative of $w$
vanishes on the straight sides of $\partial G$). Thus if we move
along this inward normal $w$ decreases and becomes negative which
is a contradiction. Therefore $w$ has all the necessary properties
of $v$ and the previous proof can be repeated (note that the comparison
must be made on $\hat{G}$, not $G$).\end{proof}
\begin{thm}
If $y_{0}$ is a nonreentrant corner, then there is $R>0$ such that
$U\cap B_{R}(y_{0})$ is elastic.%
\footnote{We need to assume that if one of the tangents to $\partial U$ at
$y_{0}$ is parallel to one of the coordinate axes, then Assumption
\ref{assu: 1} holds at $y_{0}$, and around it there is no degenerate
point of Assumption \ref{assu: 2}. This is necessary to ensure that
$\kappa_{p}$ is bounded around $y_{0}$.%
}\end{thm}
\begin{proof}
First suppose that $y_{0}$ is the intersection of two line segments
$S,S'$ of $\partial U$, making the angle $2\alpha<\pi$. We can
choose a polar coordinate centered at $y_{0}$ such that 
\begin{equation}
G:=\{(r,\theta)\,\mid\,0<r<r_{0}\,,\,-\alpha<\theta<\alpha\}\subset U.
\end{equation}
Note that 
\begin{equation}
\Delta u=\begin{cases}
-\eta & \textrm{in }E\\
\Delta d_{p} & \textrm{a.e. in }P,
\end{cases}
\end{equation}
and $\Delta d_{p}$ is bounded away from the $p\hspace{1bp}$-ridge
especially in $P$.%
\footnote{Since $\kappa_{p}$ is bounded around $y_{0}$, $1-\kappa_{p}d_{p}\to1$
uniformly, as we approach $\partial U$. Also, $1-\kappa_{p}d_{p}>0$
on $P$ around $y_{0}$, so it has a positive minimum there. In addition,
$\tau_{p}\kappa_{p}$ is bounded on $\partial U$ as we assumed that
$S_{i}$'s are smooth up to their endpoints. Hence $\Delta d_{p}$
is bounded on $P$ around $y_{0}$.%
} Thus we have that $\Delta u$ is bounded in $G$. (Note that there
is no reentrant corners there, so $\Delta d_{p}$ does not blow up
as we approach the boundary.) We assume that $r_{0}$ is small enough
so that the $p\hspace{1bp}$-closest point on $\partial U$ to any
point in $G$ lies on the sides of the angle at $y_{0}$. Then Lemma
\ref{lemma u<rd_p} implies that the region $U\cap\{r<R\}$ is elastic
for some small $R$, because 
\begin{equation}
u(r,\theta)\leq Cr^{\nu}d_{p}(r,\theta)<CR^{\nu}d_{p}(r,\theta)<d_{p}(r,\theta).
\end{equation}

Now suppose $y_{0}$ is a general nonreentrant corner. We reduce this
case to the previous case. Let $g$ be a conformal map from a sector
$G$ with straight sides to a neighborhood of $y_{0}$ (obviously
the opening angle of $G$ is the same as the one at $y_{0}$). As
proved in \citet{MR679313}, $g$ and its inverse have bounded derivative
up to the boundary. Let 
\begin{equation}
\tilde{u}(z):=u(g(z))
\end{equation}
be a function on $G$. Since $g$ is conformal, we have 
\begin{equation}
\Delta\tilde{u}=|Dg|^{2}\Delta u.
\end{equation}
And as $Dg,\Delta u$ are bounded, $\Delta\tilde{u}$ is also bounded.
Hence by Lemma \ref{lemma u<rd_p} 
\[
\tilde{u}\leq C\tilde{d}_{p}\tilde{r}^{\nu}
\]
near the vertex of $G$. Where $\tilde{d}_{p},\tilde{r}$ are respectively
the $p\hspace{1bp}$-distance from the straight sides of $G$, and
the distance from its vertex. As $Dg^{-1}$ is bounded we have $\tilde{d}_{p}\leq Cd_{p}$
and $\tilde{r}\leq Cr$. Therefore we get $u<d_{p}$ in a neighborhood
of $y_{0}$.\end{proof}
\begin{rem}
\label{remark another proof}This is another proof of the above theorem,
when $\alpha\leq\pi/4$ and $\partial U$ consists of line segments
near $y_{0}$. Here we only use Lemma \ref{lemma u<rd_p} for $\alpha>\pi/4$. 

Divide the angle by its $p\hspace{1bp}$-bisector. It is enough to
show that a sector of each of these smaller angles (the $p\hspace{1bp}$-halves)
is elastic. Let $U'$ be a domain containing $U$ such that $\partial U'$
contains $S$ and another line segment $S''$ starting at $y_{0}$
on the same side of $S$ as $S'$, making the angle $\gamma\in(\pi/2\,,\,\pi)$
with $S$. Thus $y_{0}$ is also a nonreentrant corner of $\partial U'$.
Let $w$ be the minimizer of $I$ over $K(U')$. We claim that $u\leq w$
on $U$. To see this consider the test functions 
\[
\begin{array}{c}
v:=u+(w-u)^{-}=\begin{cases}
u & u\leq w\\
w & u>w
\end{cases}\\
\\
\tilde{v}:=w+(u-w)^{+}=\begin{cases}
w & u\leq w\\
u & u>w
\end{cases}=w-(w-u)^{-}.
\end{array}
\]
(We consider $u$ to be zero outside $U$.) Now $u,w$ are both minimizers
of the functional $I$ over some closed convex sets. Also $v,\tilde{v}$
belong to those sets respectively (note that $v=0$ on $\partial U$
as $w\geq0=u$ there). Hence we have 
\begin{eqnarray*}
 &  & \int_{U}Du\cdot D(v-u)-\eta(v-u)\, dx\\
 &  & =\int_{U}Du\cdot D(w-u)^{-}-\eta(w-u)^{-}\, dx\geq0,
\end{eqnarray*}
and 
\begin{eqnarray*}
 &  & \int_{U'}Dw\cdot D(\tilde{v}-w)-\eta(\tilde{v}-w)\, dx\\
 &  & =\int_{U}-Dw\cdot D(w-u)^{-}+\eta(w-u)^{-}\, dx\geq0.
\end{eqnarray*}
Note that $(w-u)^{-}$ is zero outside $U$ as $w\geq0$ and $u=0$
there. Thus the domain of the second integral can be changed to $U$.
Adding these two integrals we get 
\begin{equation}
\int_{U}D(u-w)\cdot D(w-u)^{-}\, dx=-\int_{\{w<u\}}|D(w-u)|^{2}\, dx\geq0.
\end{equation}
Therefore $w\geq u$ a.e. which by continuity gives the required result.
Because otherwise we must have $w-u$ is a constant on the open set
$\{w<u\}$ (again using continuity) and as it vanishes on the boundary
of this set it must be zero on the set, which is a contradiction.

Hence for small $r$ 
\[
u(r,\theta)\leq w(r,\theta)<d_{p}((r,\theta)\,,\,\partial U').
\]
But when the point $(r,\theta)$ is in the $p\hspace{1bp}$-half attached
to $S$ of the angle at $y_{0}$ on $\partial U$ (and $r$ is small
enough) we have 
\[
d_{p}((r,\theta)\,,\,\partial U')=d_{p}((r,\theta)\,,\, S)=d_{p}((r,\theta)\,,\,\partial U).
\]
So $u<d_{p}(\cdot,\partial U)$ as desired.
\end{rem}

\section{Flat Boundaries}

Let $y(t)$ be a parametrization of part of $\partial U$ and as before
$y(t)+\delta(t)\mu(t)$ be the parametrization of the free boundary
attached to this part. If $\delta(t)>0$ for $t\in(a,b)$ and $\delta(a)=\delta(b)=0$
then we call the set 
\[
\{y(t)+s\mu(t)\,\mid\, t\in[a,b]\,,\, s\in[0,\delta(t)]\}
\]
a \textbf{plastic component}.
\begin{thm}
The number of plastic components attached to a closed line segment%
\footnote{This means that the segment is a closed proper subset of a line segment
of $\partial U$.%
} of $\partial U$ is finite.\end{thm}
\begin{proof}
For simplicity let the line segment be $\{(x_{1},0)\,\mid\: a\leq x_{1}\leq b\}$.
Suppose to the contrary that there are infinitely many plastic components
\[
P_{i}=\{(x_{1},x_{2})\,\mid\, x_{1}\in[a_{i},b_{i}]\,,\, x_{2}\in[0,\delta(x_{1})]\}
\]
attached to the line segment. Where $\delta$ is a continuous nonnegative
function on $[a,b]$ and $b_{i}\leq a_{i+1}$. Note that as we proved,
$\delta$ is analytic on the set $\{\delta>0\}$. Let 
\[
H_{i}:=\underset{x\in[a_{i},b_{i}]}{\max}\delta(x).
\]
Since $b_{i}-a_{i}\rightarrow0$ as $i\rightarrow\infty$, we must
have $H_{i}\rightarrow0$. Otherwise a subsequence, $H_{n_{i}}$ converges
to a positive number and by taking a further subsequence we can assume
that this subsequence is $\delta(x_{n_{i}})$ where $x_{n_{i}}\to c$.
But this contradicts the continuity of $\delta$ at $c$ because $b_{n_{i}}\to c$
too. 

Hence any line $x_{2}=\epsilon$ intersects only a finite number $n(\epsilon)$
of $P_{i}$'s, and $n(\epsilon)\rightarrow\infty$ as $\epsilon\rightarrow0$. 

Consider a piecewise analytic curve $\gamma$ in the elastic region
$E$, such that it starts at $(a,0)$ and ends at $(b,0)$ (if $P_{i}$'s
accumulate at one of these points, then $\gamma$ will start or end
at a point $(a-\epsilon,0)$ or $(b+\epsilon,0)$ if these points
have an elastic neighborhood and if not, at the free boundary point
on the lines $x_{1}=a-\epsilon$ or $x_{1}=b+\epsilon$). We also
consider $\gamma$ to be close enough to the segment so that for points
between them the $p\hspace{1bp}$-distance to $\partial U$ is the
$p\hspace{1bp}$-distance to the segment, so for those points $d_{p}(x,\partial U)$
is a function of only $x_{2}$. 

Now consider $E_{\epsilon}$ to be the elastic region enclosed by
$\gamma$ and the line $x_{2}=\epsilon$. The $\partial E_{\epsilon}$
consists of $\gamma$, segments on the line $x_{2}=\epsilon$ and
parts of $\partial P_{i}$'s. On every segment on $\partial E_{\epsilon}\cap\{x_{2}=\epsilon\}$
with endpoints on the free boundary, the function $D_{x_{1}}(u-d_{p})=D_{x_{1}}u$
is analytic and changes sign as $u-d_{p}$ is zero on the endpoints
and negative between them. Let $(c,\epsilon)\,,\,(d,\epsilon)$ for
$c<d$ be points close to those endpoints such that 
\[
D_{x_{1}}u(c,\epsilon)<0\;,\; D_{x_{1}}u(d,\epsilon)>0.
\]
Then as proved in \citet{MR679313} the level curves of the harmonic
function $D_{x_{1}}u$ can be continued until they exit its domain
$\cup E_{\epsilon}$. But $D_{x_{1}}u=D_{x_{1}}(u-d_{p})$ is zero
on the free boundary and on the segment $x_{2}=0$, hence they must
exist through $\gamma$. This implies that $D_{x_{1}}u$ changes sign
at least $2n(\epsilon)-1$ times on $\gamma$, since the level curves
do not cross each other. But this number grows to infinity as $\epsilon\rightarrow0$
contradicting the fact that $\gamma$ is piecewise analytic and $D_{x_{1}}u$
is analytic on a neighborhood of it.\end{proof}
\begin{acknowledgement*}
I would like to express my gratitude to Lawrence C. Evans and Nicolai
Reshetikhin for their invaluable help with this research.
\end{acknowledgement*}
\bibliographystyle{plainnat}
\bibliography{New-Paper-Bibliography}

\vspace{0.5cm}

\address{Department of Mathematics, UC Berkeley, Berkeley, CA 94720, USA}

\email{safdari@berkeley.edu}
\end{document}